\documentclass[leqno,12pt]{article}

\usepackage{amsmath}
\usepackage{amscd}
\usepackage{amsopn}
\usepackage{amsthm}
\usepackage{amsfonts}
\usepackage{amssymb}
\usepackage{amsfonts}
\usepackage{bbm}
\usepackage{latexsym}

\usepackage{color}

\makeatletter
\@addtoreset{equation}{section}
\makeatother

\newtheorem{lemma}{LEMMA}[section]
\newtheorem{proposition}[lemma]{PROPOSITION}
\newtheorem{corollary}[lemma]{COROLLARY}
\newtheorem{theorem}[lemma]{THEOREM}
\newtheorem{remark}[lemma]{REMARK}
\newtheorem{remarks}[lemma]{REMARKS}
\newtheorem{example}[lemma]{EXAMPLE}

\newtheorem{definition}[lemma]{DEFINITION}
\newtheorem{assumption}[lemma]{ASSUMPTION}

\newcommand{\real}{\mathbbm{R}}
\newcommand{\nat}{\mathbbm{N}}

\newcommand{\limn}{\lim_{n \to \infty}}
\newcommand{\sumn}{\sum_{n\ge 1}}
\newcommand{\sumk}{\sum_{k\ge 1}}
\newcommand{\sumnn}{\sum_{n\ge 0}}               

\newcommand{\g}{\gamma}

\newcommand{\vp}{\varphi}
\newcommand{\ve}{\varepsilon}

\newcommand{\reald}{{\real^d}}

\newcommand{\on}{\quad\text{ on }}
\newcommand{\und}{\quad\mbox{ and }\quad}

\newcommand{\ov}{\overline}

\newcommand{\V}{\mathcal V}  
\newcommand{\W}{\mathcal W}  
\newcommand{\R}{\mathcal R}  
\newcommand{\N}{\mathcal N}  
\newcommand{\C}{\mathcal C}  

\newcommand{\F}{\mathcal F}

\newcommand{\B}{\mathcal B}

\newcommand{\M}{\mathcal M}

\newcommand{\U}{{\mathcal U}}

\newcommand{\itemframe}%
{\setlength{\parskip}{10pt}\begin{enumerate} \setlength{\topsep}{10pt}%
\setlength{\itemsep}{15pt}\setlength{\parsep}{5pt}}

\newcommand{\vx}{\ve_x}

\newcommand{\vc}{{V^c}}

\newcommand{\es}{\mathcal E_{\mathfrak X}}           

\newcommand{\bu}{\B^\ast}
\newcommand{\bup}{\widetilde{\B^\ast}}

\date{}

 \title{Nearly hyperharmonic functions are infima\\ of excessive  functions}
 
\author{Wolfhard Hansen and Ivan Netuka%
\thanks{The research of the second author was supported by
CRC 1283 of the German Research Foundation.}}

\begin{document}

\maketitle

\begin{abstract}
Let $\mathfrak X$ be a Hunt process on a locally compact space $X$ such that the set~$\mathcal E_{\mathfrak X}$ 
of its Borel measurable excessive functions separates points, 
every function in~$\mathcal E_{\mathfrak X}$
is the supremum of its continuous minorants in $\mathcal E_{\mathfrak X}$ and there are strictly positive 
continuous functions $v,w\in\mathcal E_{\mathfrak X}$ such that $v/w$ vanishes at infinity.  

A numerical function
$u\ge 0$ on $X$ is said to be \emph{nearly hyperharmonic},  \hbox{if $\int^\ast u\circ X_{\tau_V}\,dP^x\le u(x)$}
for every $x\in X$ and every  relatively compact open neighborhood $V$ of $x$, where $\tau_V$
denotes the exit time of $V$. For every such function $u$, its  lower semicontinous regularization
$\hat u$ is excessive.

The main purpose of the paper is to give a   short, complete and understandable proof for the statement
that $ u=\inf \{w\in \mathcal E_{\mathfrak X}\colon w\ge u\}$ 
for every Borel measurable nearly hyperharmonic function on $X$. 
Principal 
 novelties of our approach are the following: 

1. A quick reduction to the   special case,  
where starting at  $x\in X$ with $u(x)<\infty$  
 the  expected number of times  the  process $\mathfrak X$ visits  the 
  set of  points $y\in X$, where $\hat u(y):=\liminf_{z\to y} u(z)<u(y)$, is finite.
 
2. The consequent use of (only) the strong Markov property. 

3. The proof of the equality 
$\int u\,d\mu=\inf\{\int w\,d\mu\colon w\in \es,\,w\ge u\}$ not only for
measures $\mu$ satisfying $\int w\,d\mu<\infty$ for some excessive majorant~$w$ of~$u$,   
but also for all finite measures.

At the end, the measurability assumption on $u$ is weakened considerably. 

\bigskip



\end{abstract}

\section{Main result}\label{main-result}

Let $X$ be a locally compact space with countable base, let $\B$ denote the $\sigma$-algebra of all
Borel sets in $X$, and let~$\B(X)$, $\C(X)$ respectively be the set of all numerical  functions on $X$
which are Borel measurable, continuous and real respectively. As usual, given a set $\F$ of functions 
on $X$, a superscript ``+'',  a subscript ``$b$'' respectively will indicate that we consider functions in $\F$
which are positive, bounded respectively.
Let $\M(X)$ denote the set of all positive (Radon) measures on $X$.

Let $\mathfrak X=(\Omega, \mathfrak M, \mathfrak M_t, X_t,\theta_t, P^x)$ be a
Hunt process on $X$ (see \cite[p.\,45]{BG}). 
Let $\mathbbm P=(P_t)_{t>0}$ denote the transition semigroup of $\mathfrak X$, that is, 
$P_tf(x)=E^x(f\circ X_t) $ for all  $t>0$, $f\in \B^+(X)$ and $x\in X$.

We assume that the Hunt process $\mathfrak X$ is \emph{nice} in the following sense.
Its set 
\begin{equation}\label{Hunt} 
\es:=\{w\in \B^+(X)\colon \sup_{t>0}P_tw=w\}
\end{equation} 
of (Borel measurable) excessive functions has the following properties:
\begin{itemize} 
\item[\rm (C)] Continuity:
Every $w\in \es$ is the supremum of its minorants in $\es\cap\C(X)$.
\item[\rm (S)] Separation:
$\es$ is  linearly separating, that is, for all $x\ne y$ and  $\g>0$, there exists a function $w\in \es$ 
such that  $w(x)\ne \g w(y)$.
\item[\rm (T)] Transience:
There are  strictly positive functions $v,w\in\es\cap \C(X)$ such that the quotient  $v/w$ 
tends to $0$  at~infinity.
\end{itemize} 

Let us observe that (C) trivially holds if the kernels $P_t$, $t>0$, or  at least 
the corresponding resolvent kernels $V_\lambda:=\int_0^\infty e^{-\lambda t}P_t\,dt$, $\lambda>0$, 
  are strong Feller, that is,
 map $\B_b(X)$ into $\C_b(X)$. 

For every set $A$ in $X$,   the first entry time $D_A$ and   the first hitting time $T_A$
are defined for $\omega\in\Omega$ by
\begin{equation*}
 D_A(\omega):=\inf\{s\ge 0\colon X_s(\omega)\in A\} \und
        T_A(\omega):=\inf\{s> 0\colon X_s(\omega)\in A\}.
\end{equation*} 
 Let $\U_c$ be the set of all relatively compact open sets $V$ in $X$, $V^c:=X\setminus V$.
A~numerical function $u\ge 0$ is called \emph{nearly hyperharmonic} if          
\begin{equation}\label{def-nearly}
                  \int^\ast u\circ X_{D_\vc} \,dP^x\le u(x) \quad \mbox { for all $ x\in X$ and
                    neighborhoods $V\in \U_c$ of $x$}.
\end{equation} 

Clearly, the set $\N$ (denoted by $\N^+$ in \cite{HN-mertens}) of such functions 
is a convex cone which contains $\es$ and is stable under increasing limits and 
\emph{arbitrary} infima. Moreover, it contains \emph{every} numerical function $u\ge 0$   
which vanishes outside a set~$E$ which is polar, that is, satisfies $T_E=\infty$ almost surely. 
For space-time Brownian motion on $\reald\times \real$, 
every function $u\colon \reald\times \real\to [0,\infty]$ satisfying $u(x,t)\le u(x',t')$,
whenever $t\le t'$, is nearly hyperharmonic.

The   purpose of this paper is to give a short,  complete and understandable proof 
for the following statement (where the implications (3)\,$\Rightarrow$\,(2)\,$\Rightarrow$\,(1)
hold trivially). 

\begin{theorem}\label{main}
For every $u\in \B^+(X)$  the following statements are equivalent:
\begin{itemize}
\item[\rm (1)] 
The function $u$ is nearly hyperharmonic.
\item[\rm (2)]
The function $u$ is the infimum of its excessive majorants.
\item[\rm (3)] For all $\mu\in \M(X)$  
such that   $\mu(A) +\int_{X\setminus A}  w\,d\mu<\infty$   
for some $A\in\B$ and majorant $w\in\es$ of $u$, 
\begin{equation}\label{int-inf}
     \int u\,d\mu=\inf\{\int w\,d\mu\colon w\in\es, \ w\ge u\}.
\end{equation}  
\end{itemize} 
In particular, for every $\vp\in\B^+(X)$, 
the function $R_\vp:=\inf\{w\in\es\colon w\ge \vp\}$ 
is the smallest nearly hyperharmonic majorant of $\vp$.
\end{theorem} 
  
\begin{remark}\label{KK}
{\rm 
Of course, (\ref{int-inf}) trivially holds if $\int u\,d\mu=\infty$ (we take   $w=\infty $).
Since $\es$ is $\wedge$-stable, the set of all  $\mu\in \M(X)$
satisfying (\ref{int-inf})  is a convex cone.
 If~$\ve>0$ 
and $A\in \B$ such that
 $\mu(A)+\int_{X\setminus A} w\,d\mu<\infty$,   
there is a union $A'$ of $A$ with a  compact in $X\setminus A$ such that $\mu(A')<\infty$ and $\int_{X\setminus A'} w\,d\mu<\ve$.
}
\end{remark}

In fact, we shall finally prove that, for functions $u\colon X\to [0,\infty]$,
 the  equivalence (1)\,$\Leftrightarrow$\,(3) 
already holds if $u$ is nearly Borel measurable (Theorem \ref{nearly-Borel}) and that
(1)\,$\Leftrightarrow$\,(2) even holds if $u$ is only supposed  
to be equal to a Borel measurable function 
outside a polar set (Theorem \ref{tilde-Borel}). 
Moreover, assuming that $u$ is nearly hyperharmonic and equal to a universally 
 measurable function outside a polar set,
we characterize the validity of (2) in various ways
(Corollary~\ref{general-2}).

Analogous statements can be found for different settings and functions, which there are called 
\emph{strongly supermedian},  
 in \cite{mertens,feyel-rep,
feyel-fine, beznea-boboc-feyel, beznea-boboc-book}, 
but the proofs given therein seem to be either incomprehensible 
or incomplete  (see \cite{mertens, feyel-rep,feyel-fine}) or, as in  \cite{beznea-boboc-feyel} and 
\hbox{\cite[Section 4] {beznea-boboc-book}},
very long and delicate.

The main novelties of our approach are 
\begin{itemize}
\item  
the insight that  for a proof of   inequalities $u\ge \eta \inf\{w\in\es\colon w\ge u\} $ for  
$\eta\in (0,1)$ it suffices to consider the
special case, where starting in~$\{u<\infty\}$ 
  the expected number of times
  the process $\mathfrak X$ visits
 the  set~of  points~$y\in X$, where $\hat u(y):=\liminf_{z\to y} u(z)<u(y)$,
is finite.
\item
the consequent use of (only) the strong Markov property,  
\item
the verification of the equalities in (2) and (3) first for nearly hyperharmonic functions
 $u\in\B^+(X)$ having a certain finiteness property, which then implies~(!) that    
\emph{every} nearly hyperharmonic   $u\in\B^+(X)$
has this property,
\item
the equality  (\ref{int-inf})  not only for 
measures $\mu$ satisfying $\int w\,d\mu<\infty$ for some excessive majorant $w$ of $u$, 
but also for all finite measures~$\mu$.

\end{itemize}

Let us observe that the additional statement 
 in Theorem \ref{main} is not only of interest 
in its own right, 
but also because of the following consequence (see \cite[Propositions~2.4, 2.5 and Theorem 3.1]{HN-mertens}).

\begin{corollary}\label{cons}
Let $\vp\in\B^+(X)$. Then $R_\vp=\vp\vee \widehat{R_\vp}\in \B^+(X)$ and
\begin{equation*}
        R_\vp(x)=\sup\{\int \vp\circ X_{D_\vc}\,dP^x\colon x\in V\in \U_c\}, \qquad x\in X.
\end{equation*} 
\end{corollary}

In Section \ref{sec-prelim} we discuss the close relationship between nice Hunt processes and balayage spaces
and  establish a~crucial inequality for nearly hyperharmonic functions (Lemma \ref{PFuh}).
In Section \ref{sec-special}  
we treat   the special case indicated above. 
In Section~\ref{sec-general} we shall see very quickly that the equality $R_u=u$ for arbitrary nearly hyperharmonic functions 
$u\in \B^+(X)$  is a~consequence of our result for the special case and yields the additional statement in Theorem \ref{main}. 
The implication (1)\,$\Rightarrow$\,(3)  
is derived   in Section~\ref{remaining}, and  
 in Section \ref{weaker} we present our results under weaker measurability assumptions.
 In Section \ref{standard} we briefly indicate the use of our approach in the
 general setting of right processes.

\section{Preliminaries}\label{sec-prelim}

Let us first recall the following. Let $\W$ be any
convex cone  of positive numerical functions on $X$ having the properties stated in  (C), (S)
and  (T) for $\es$ (so that every function 
in $\W$ is lower semicontinuous). 
The ($\W$-)\emph{fine topology} on $X$ is the coarsest topology on $X$ which is at least as fine
as the initial topology and such that every function in $\W$ is continuous. Given  
$\vp\colon X\to [0,\infty] $, 
 let $\hat \vp$, $\hat \vp^f$ resp.\ denote the largest lower semicontinuous,
finely lower semicontinuous resp.\ minorant of~$\vp$.

Then  $(X,\W)$ is called a \emph{balayage space} provided the following hold
(see \cite{BH,H-course} and \cite[Appendix 8.1]{HN-unavoidable-hameasures}): 
\begin{itemize} 
\item[\rm (i)]
If $v_n\in \W$, $v_n\uparrow v$, then $v\in \W$.
\item[\rm (ii)]
If $\V\subset \W$, then $\widehat{\inf \V}^f\in\W$. 
\item[\rm (iii)]
If $u,v',v''\in\W$,  $u\le v'+v''$, then  there exist $u',u''\in\W$ such that  $u=u'+u''$ and $u'\le v'$,  $u''\le v''$. 
\end{itemize}

By \cite[II.4.9]{BH} (see also \cite[Corollary 2.3.8]{H-course}), for our nice Hunt process
$\mathfrak X$,  the pair $(X,\es)$ is a balayage space (of course, $\lim_{t\to 0} P_tf=f$ for
every $f\in \C_b(X)$ by right continuity of the paths).  
So we may use results obtained in \cite{BH} and in the recent paper \cite{HN-mertens}.

\begin{remarks} {\rm
1. We note that, conversely, for every balayage space $(X,\W)$ with $1\in\W$, there exists a~corresponding 
nice Hunt process (see \cite[IV.8.1]{BH}). For that matter, the condition \hbox{$1\in\W$} is not really restrictive
since, given any balayage space $(X,\W)$, the standard normalization $\widetilde \W:=(1/\tilde w)\W$ with 
any strictly positive $\tilde w\in \W\cap \C(X)$ leads to a~balayage space~$(X,\widetilde \W)$
with $1\in \widetilde \W$.

2. A characterization by harmonic kernels reveals that the notion of a balayage space generalizes the 
notion of a $\mathcal P$-harmonic space. Therefore the theory of balayage spaces is known to 
cover the potential theory for very general partial  differential operators of second order
(see, for instance,  \cite{GH1}). 
}
\end{remarks}
 
Of course, for any numerical function $\vp\ge 0$ on $X$,   
\begin{equation*}
                         R_\vp:=\inf\{w\in\es\colon w\ge \vp\} \in \N,
\end{equation*} 
and, by property (ii) of  balayage spaces, 
\begin{equation}\label{red-es} 
                              \hat R_\vp:=\widehat{R_\vp}=\widehat{R_\vp}^f\in \es.
\end{equation} 
By \cite[Proposition 2.2 and p.\,6]{HN-mertens}, we know even that   $\hat u=\hat u^f\in \es$ for all $u\in \N$.


We recall that, for  arbitrary subsets $A$ of $X$ and $w\in \es$, 
\begin{equation*} 
                                R_w^A:=R_{1_Aw}=\inf\{v \in \es\colon v\ge w \mbox{ on }A\} \und 
                                 \hat R_w^A:=\widehat{R_w^A}, 
\end{equation*} 
leading to measures $\ve_x^A$ and  $\hat\ve_x^A$ on $X$ which are  characterized by
\begin{equation*}
               \int w \,d\ve_x^A=R_w^A (x) \und                \int w \,d\hat \ve_x^A=\hat R_w^A (x), \qquad w\in \es;
\end{equation*}
 see \cite[II.4.3, II.5.4 and VI.2.1]{BH} 
 (in \cite{BH} these measures are denoted by~$\overset\circ  \ve{}_x^A$ and $\vx^A$). 
For later use,  let us observe that trivially
\begin{equation}\label{Rvv}
  R_v=v, \quad\mbox{whenever $v=R_w^A$,  $w\in \es$}.
  \end{equation}          

 Given a stopping time $T$, we define 
as usual\footnote{We tacitly assume
 that we have an isolated point $\Delta$ added to $X$, that functions on $X$ are
identified with functions on $X_\Delta:=X\cup\{\Delta\}$ vanishing at $\Delta$ and that $X_t\colon 
[0,\infty] \to X_\Delta$ with $X_\infty=\Delta$ and $X_t(\omega)=\Delta$, whenever $t\ge s$ and
$X_s(\omega)=\Delta$.}
\begin{equation*}
P_Tf(x):=E^x(f\circ X_T) \quad\mbox{ for all $f\in \B^+(X)$ and $x\in X$}.
\end{equation*} 
Suppose for the moment that $A\in \B$. Then, by \cite[VI.3.14]{BH},
 both $D_A$ and $T_A$ are stopping times
 and, 
 writing $P_A$ instead of $P_{D_A}$ and $\hat P_A$ instead of  $P_{T_A}$,   
\begin{equation}\label{PPhat}
P_Aw =R_w^A \und 
 \hat P_A w =\hat R_w^A
\end{equation}
for every $w\in\es$  (cf. \cite[6.12]{BG}),  
where obviously   
\begin{equation}\label{PDT}
  \mbox{$P_Aw=w$ on $A$} \und \mbox{$P_Aw=\hat P_Aw$ on $X\setminus A$.}
  \end{equation} 
  By (\ref{PPhat}) and \cite[Section II]{BH},  
                             $P_Af, \hat P_Af \in \B(X) $  for all $f\in \B^+(X)$. 
                       Hence $P_A$ and~$\hat P_A$ are kernels on $X$. 
   Moreover,   
  \begin{equation*} 
P_A(x,B)=\ve_x^A(B) \und \hat P_A(x,B)=\hat \ve_x^A(B)
\end{equation*} 
 for all $x\in X$ and $B\in \B$. In particular, our definition of nearly hyperharmonic functions by (\ref{def-nearly})
coincides with the definition given by \cite[(2.2)]{HN-mertens}.

The following  simple stability result will be useful.  

\begin{lemma}\label{sums}
For every $\mu\in\M(X)$, the set $\F$   of all functions $f\in\B^+(X)$     
such that $\int f\,d\mu=\inf\{\int w\,d\mu\colon w\in\es,\, w\ge f\}$
is a convex cone which is closed under countable sums.
\end{lemma}

\begin{proof} Of course, $0\in \F$ and  $a f\in\F$ for all $a> 0$ and $f\in \F$.
Let $(f_n)$ be a~sequence in~$\F$ and $f:=\sumn  f_n$ such that  $\int f\,d\mu<\infty$.
 Given $\ve>0$, we may choose  $w_n\in\es$,
$n\in\nat$, such that $\int w_n\,d\mu<\int f_n\,d\mu+ 2^{-n}\ve$. Then $w:=\sumn w_n\in\es$,
$w\ge f$ and $\int w\,d\mu< \int f\,d\mu+\ve$. 
\end{proof}

We recall that an arbitrary set $A$ in $X$ is
called \emph{thin} at a point $x\in X$ if $\hat \ve_x^A\ne \vx$. 
By definition, the \emph{base} $b(A)$ of $A$ is the set of all $x\in X$ such that $A$ is not thin at~$x$, that is, $\hat
\ve_x^A=\vx$. By \cite[VI.4.8]{BH},  
\begin{equation}\label{base-prob}
                                    b(A)=\{x\in X\colon T_A=0\mbox{ $P^x$-almost surely}\}, \quad \mbox{ if } A\in\B.
\end{equation} 
  By \cite[VI.4.1 and VI.4.4]{BH}, the base of every set $A$ in $X$ is a finely closed $G_\delta$-set
containing the fine interior of $A$, and $A\cup b(A)$ is the fine closure of $A$.
 Moreover, for every $x\in X$, the measure $\hat\ve_x^A$ is supported by the fine closure of $A$, 
that is, the inner measure of its complement is zero; see \cite[VI.4.6]{BH}. 

A set $F$ in $X$ is called \emph{totally thin} if $b(F)=\emptyset$ so that, in particular, $F$ is finely closed. 
A \emph{semipolar set} is a countable union of totally thin sets. 
We know that, for any infimum $u$ of functions in $\es$, the set  $\{\hat u<u\} $ is
semipolar; see \cite[VI.5.11] {BH}. 
 
\begin{example}\label{hyperplanes}{\rm     
For space-time Brownian motion on $\reald\times \real$, every hyperplane $H_t:=\real^d\times \{t\}$ is totally thin.
 
}\end{example}
 
For the remainder of this section let us fix a function $u\in \B^+(X)$ which is nearly hyperharmonic.
By  \cite[Proposition 2.5]{HN-mertens}, for every $A\in \B$,
\begin{equation}\label{nearly-used}
          P_Au\le u  \und \hat P_A u\le u. 
\end{equation}

\begin{remark} {\rm         
Let $S,T$ be stopping times for $\mathfrak X$,  $S\le T$.  Then
$P_Tw\le P_Sw\le w$ for every $w\in\es$; see \cite[VI.3.4]{BH}.
By \cite[Corollary 2.6]{HN-mertens} (which uses that, for   $x\in X$,
 the extreme points in the weak$^\ast$-compact convex set $\M_x(\es)$ 
of all measures~$\mu$ on~$X$ satisfying \hbox{$\int w\,d\mu\le w(x)$} for every $w\in\es$
are the measures $\vx^A$, $A\in \B$),  this implies that $P_Su\le u$.  So the nearly hyperharmonic 
function~$u$  is strongly                                
supermedian in the sense of 
\hbox{\cite{beznea-boboc-feyel, beznea-boboc-book, feyel-rep, feyel-fine, moko-ens-compacts}}. 
 By~\cite[Proposition 2.7]{HN-mertens}, 
we even get the  inequality $P_Tu\le P_Su$ (and hence, by a standard argument, 
$ E^x(u\circ X_T|\mathfrak M_S)\le u\circ X_S$ $P^x$-a.s.\ for every   $x\in X$).

Since we shall not use these facts in the sequel, they may also be viewed
as \emph{consequences}  of Theorem \ref{main}      
  (to obtain $P_T(u\wedge n) \le P_S(u\wedge n) $, $n\in\nat$, 
we  consider the finite measures $\mu:=P_T(x,\cdot)+P_S(x,\cdot)$, $x\in X$). 
}
\end{remark} 

Let us note  that (\ref{nearly-used}) implies the following.

\begin{lemma}\label{fusc}
The function $u$ is finely upper semicontinuous and, starting in the finely open set $U:=\{u<\infty\}$,
the process $\mathfrak X$ does not leave $U$, that is,
\begin{equation}\label{U-stable}
P^x[T_{X\setminus U}<\infty]=0 \qquad\mbox{ for every $x\in U$.}
\end{equation} 
\end{lemma} 

\begin{proof} 
Let $a\in [0,\infty]$, $A:=\{u\ge a\}$ and $x\in X\setminus A$. 
 By \cite[VI.3.14]{BH}, 
there exists an increasing sequence $(K_n)$ of  compacts in the Borel set $A$ 
such that $T_{K_n}\downarrow T_A$~$P^x$-a.s. Since $X_{T_{K_n}}\in K_n$ on
$[T_{K_n}<\infty]$, the inequalities $P_{K_n}u\le u$ yield that 
\begin{equation*} 
      aP^x[T_A<\infty] =\limn aP^x[T_{K_n}<\infty]\le u(x)<a. 
\end{equation*} 
Hence $P^x[T_A<\infty]<1$, $x\notin b(A)$. So $A$ is finely closed
showing that  $u$ is finely upper semicontinuous. 
Finally,  taking   $a=\infty$, we see that (\ref{U-stable}) holds.
\end{proof}

For every $V\in \U_c$, due to the lower semicontinuity  
of  $P_\vc u$ on $V$, we know that 
\begin{equation}\label{nearly-used-l}
  \hat P_\vc u=P_\vc u\le \hat u \on V
\end{equation} 
 (see \cite[(2.3)]{HN-mertens}).
The following  more general estimate will be crucial in Section~\ref{sec-special}.

\begin{lemma}\label{PFuh} 
Let   $A\in\B$ and $x\in X\setminus b(A)$ such that $x$ is not finely isolated.  Then  
\begin{equation*} 
\hat P_Au (x) =E^x(u\circ X_{T_A}) \le \hat u (x). 
\end{equation*}  
\end{lemma}

\begin{proof} 
Since $u_n:=u\wedge n$ is nearly hyperharmonic  for every $n\in\nat$ and
  $\hat u_n\uparrow \hat u$ by \cite[Proposition 2.3]{HN-mertens},
  we may assume without loss of generality that $u$ is bounded, say
  $u< M<\infty$. 

Let $W_n\in \U_c$ such that $W_n\downarrow \{x\}$ as $n\to\infty$.
 By assumption, $x\in b(X\setminus \{x\})$, 
 and hence $D_{W_n^c}\downarrow
D_{X\setminus \{x\}}=0$ $P^x$-a.s., whereas $T_A>0$ $P^x$-a.s.
So there exists $n\in\nat$ such that $\tau:=D_{W_n^c}$ satisfies
$      P^x[T_A\le \tau ]<\ve/M$, and therefore 
\begin{equation*}
  E^x(u\circ X_{T_A}) \le \ve +  E^x(u\circ X_{T_A}; T_A>\tau). 
\end{equation*} 
Let us note that $\tau>0$ $P^x$-a.s. and that 
obviously, on the set $[T_A>\tau>0]$ we have $ T_A=\tau + D_A\circ \theta_\tau$
whence $X_{T_A}=X_{D_A}\circ  \theta_\tau$. Thus 
we conclude that
\begin{eqnarray*} 
 E^x(u\circ X_{T_A}; T_A>\tau)&\le& E^x(u\circ X_{D_A}\circ
 \theta_\tau)         =E^x(E^{X_\tau}(u\circ X_{D_A}))\\[2pt]
&=&P_\tau P_Au(x)\le P_\tau u(x) \le \hat u(x),
\end{eqnarray*} 
by the strong Markov property and (\ref{nearly-used-l}).
\end{proof}

Finally, let us recursively define stopping times $S_n^A$, $A\in \B$, $n\ge 0$, by
\begin{equation}\label{def-Sn}
   S_0^A:=D_A \und S_{n+1}^A:=S_n^A+T_A\circ \theta_{S_n^A}.
\end{equation} 
So  $S_{n+1}^A$ is the time of the first hitting of $A$ after the time $S_n$,  
  $X_{S_{n+1}^A}=X_{T_A}\circ \theta_{S_n^A}$; see \cite[Section IV.6]{BH}.

\begin{proposition} \label{Proposition-PAn}  
             Let $A\in \B$ and $n\ge 0$. Then
\begin{equation} \label{PAn}
  P_n^Af(x):=E^x(f\circ X_{S_n^A}) = P_A(\hat P_A)^n f(x) , \qquad   f\in\B^+(X), \ x\in X. 
\end{equation}
\end{proposition}  

\begin{proof} 
 Obviously,  (\ref{PAn}) holds for $n=0$. 
  Suppose that it  is true for some $n\ge 0$, 
  and let $f\in \B^+(X)$, $x\in X$,  $S_n:=S_n^A$.    
  By the  strong Markov property,
\begin{eqnarray*} 
  E^x( f\circ X_{S_{n+1}})= E^x( f\circ X_{T_A}\circ \theta_{S_n})    &      =&
                              E^x (E^{X_{S_n}}(f\circ T_A))\\  &=&E^x(( \hat P_A f)\circ X_{S_n})
                                                                  = P_A(\hat P_A)^{n+1} f(x).
          \end{eqnarray*} 
 \end{proof} 

  Given $A\in \B $,  let  
\begin{equation*}
  P^A:=\sumnn P_n^A
\end{equation*}
so that, by Proposition \ref{Proposition-PAn},
\begin{equation*}
  P^A1(x)= \sumnn P^x[S_n^A<\infty]=\sumnn P_A(\hat P_A)^n1(x)
  \end{equation*} 
  is the expected number of times
the process~$\mathfrak X$ visits $A$ when  starting in~$x$. 
Clearly, $ P^A1=P_Aw$ with $w:=\sumnn (\hat P_A)^n1\in\es$  and $\hat P_A w+1=w$.
Thus, by~(\ref{Rvv}), (\ref{PPhat}) and (\ref{PDT}), we obtain the following
(where  we may note that $\hat v$   is the expected number of visits
 of $A$ at \emph{strictly positive} times). 

\begin{proposition}\label{Rvav}
  For every $A\in \B$,   $v:=P^A1$ satisfies $ R_v=v$  and $\hat v+1_A=v$.
\end{proposition}

\section{A special case}\label{sec-special}

Throughout this section we fix a 
nearly hyperharmonic function $u\in\B^+(X)$ and suppose the following. 

\begin{assumption}\label{ass} 
  Defining  $U:=\{u<\infty\}$,        
and  $F:=\{\hat u<u\} $, we have
\begin{equation}\label{ass-formula}
       P^F 1<\infty  \on U,
     \end{equation}
     that is, starting in $U$, the expected number of times the process $\mathfrak X$ visits~$F$ is finite.
 \end{assumption}

If $x\in b(F)$, then, by (\ref{base-prob}),  $S_n=0$ $P^x$-a.s.\ for every $n\ge 0$, $P^F1(x)=\infty$.  
Hence  (\ref{ass-formula}) implies that $b(F)\cap U=\emptyset$. So $F$ is   totally thin
if $u<\infty$.
By Lemma \ref{fusc}  and a~straightforward induction, we see  that, for every $x\in U$ and 
 $P^x$-a.e.\ $\omega\in [S_n<\infty]$,  
 \begin{equation}\label{FcapU}
      X_{S_n}(\omega)\in F\cap U   \quad and \quad    S_n(\omega)<S_{n+1}(\omega),
\end{equation} 
 and, for $P^x$-a.e.\  $\omega\in\Omega$,
\begin{equation}\label{sge0}
             \{S_n(\omega) \colon n\ge 0, \, S_n (\omega)<\infty\}=\{s\ge 0\colon X_s(\omega)\in F\}.
\end{equation} 

\begin{example}{\rm
Let us consider space-time Brownian motion 
on $\reald\times \real$ and fix a~sequence $(t_n)$ in $\real$.
For   $n\in\nat$, let $A_n$ be an arbitrary subset of~$H_{t_n}=\reald\times \{t_n\}$ and let 
  \begin{equation*} 
 v_n:=1_{A_n}+1_{\reald\times (t_n,\infty)}.
 \end{equation*} 
 Then $v:=\sumn 2^{-n}v_n\in\N$, $0\le v\le 1$,  and $\{\hat v<v\}$ is the union of all $A_n$, $n\in\nat$.

 If $\{t_n\colon n\in\nat\} $ is dense in~$\real$  
 and $A_n=H_{t_n}$, $n\in\nat$, then
 $\{\hat v <v\}$ is finely dense in~$\reald\times \real$, 
 $P^{\{\hat v<v\}}1=\infty$. 
So Assumption \ref{ass} is very restrictive. 

If  $t_n=-1/n$ and $A_n=H_{t_n}$, $n\in\nat$, then $v:=\infty\cdot 1_{\reald\times [0,\infty] }
+\sum_{n\ge 1} 2^{-n}v_n \in \N$ and $v$ satisfies  Assumption \ref{ass} with $b(\{\hat v<v\})=H_0$.
 }
\end{example}

We define  functions $g$ and $u_0$  on $X$ by 
\begin{equation*} 
  g(x):=   
   P^F(1_Uu-1_U\hat u)\und u_0(x):=u(x)-g(x), \qquad x\in U,
\end{equation*} 
  and $g(x)=u_0(x)=\infty$, if $ x\in X\setminus U$.

\begin{definition}   
For   $A\in \B$,  let $\R(A)$ be the set    
of  sums of a function in  $\es$  
  and    countably many functions~$P_Bw$,   
where $B\in \B$, $w\in \es$, $B\subset A$ and $w$ is~bounded. 

 \end{definition}

Clearly, $\R(A)\subset  \R(X)\subset \N\cap\B^+(X)$  and,
by   (\ref{Rvv}) and Lemma~\ref{sums},   
\begin{equation}\label{NR}
            v=R_v=\inf\{w\in \es\colon w\ge v\} \quad\mbox{ for every }v\in \R(X).
\end{equation}

In this section we shall establish  the following result. 

\begin{theorem}\label{main-special}
 The function $g$ is a minorant of $u$, both $g$ and $u_0$ are nearly hyperharmonic, $g-\hat g=u-\hat u $ on $U$
and $\hat u_0=u_0$ on $U$.

Further,   there are functions $g_1, u_1\in \R(F)$ 
 such that $g_1=g$ on~$U$ and $u_1=u$ on~$U$. In particular, if $u<\infty$, then $u_0\in \es$ and $u=R_u$.
\end{theorem} 
 
We prepare its proof   with a lemma leading to estimates by telescoping series.   

\begin{lemma}\label{g-tau}
 Let $V\subset X$ be   open,   
    $x\in V\cap U$, $\tau:=D_\vc$ and $\Omega_n:=\{S_n<\tau\}$. 
Then, for all $n\ge 0$,
 \begin{equation}\label{tel-formula} 
E^x(\hat u\circ X_{S_n} ; \Omega_n)\ge E^x( u\circ X_{S_{n+1}} ; \Omega_{n+1})
                                                         +E^x(u\circ X_\tau;\Omega_n\setminus \Omega_{n+1}).  
\end{equation} 
\end{lemma} 

\begin{proof}   Let  $T:= T_{F\cup \vc} $.
    By Lemma~\ref{PFuh}, 
    $E^y(u\circ X_T)\le \hat u(y)$ for every $y\in F\cap U$,
    and hence, by (\ref{FcapU}) and the strong Markov property, 
\begin{equation*} 
  E^x(\hat u\circ X_{S_n}; \Omega_n) \ge    E^x(E^{X_{S_n}} (u\circ X_T); \Omega_n)
  =E^x(u\circ X_T\circ \theta_{S_n}; \Omega_n)
\end{equation*} 
for every $n\ge 0$.
The proof is completed observing   
that $S_n+T\circ \theta_{S_n}=S_{n+1} $ on $\Omega_{n+1}$
and $S_n+T\circ \theta_{S_n}= \tau$   on~$\Omega_n\setminus \Omega_{n+1}$.
\end{proof}

\begin{proposition}\label{g-function} 
The function $g$ is a minorant of $u$. 
\end{proposition}

\begin{proof} 
Let $x\in U$. By  Lemma \ref{g-tau}      with $V:=X$, $\tau=\infty$,  
\begin{equation*} 
g(x)\le  \sumnn\bigl( E^x(u\circ X_{S_n})-E^x(u\circ X_{S_{n+1}})\bigr)   \le  E^x(u\circ X_{S_0})= P_F u(x),
\end{equation*} 
where $P_F u(x)\le u(x)$, by (\ref{nearly-used}). 
\end{proof}

Finally,   we need the  following.   

\begin{lemma}\label{tech}
  
 Let $E$ be a Borel subset of $F$.   Then $v:= P^E1=P^F1_E$  and
 \begin{equation}\label{markov}
   P_Bv(x) = E^x(\sumnn 1_{E}\circ X_{S_n}\, 1_{[S_n\ge D_B]})\quad\mbox{for all $B\in \B$ and $x\in X$. }
\end{equation} 
\end{lemma}

\begin{proof}  
 Let $T_n:=S_n^E$, $n\ge 0$. By Proposition \ref{PAn}  and the strong Markov property,
 \begin{eqnarray*} 
&&P_Bv(x)=E^x(E^{X_{D_B}}(\sumnn 1_{E}\circ X_{T_n}))\\
           & =&E^x(\sumnn 1_{E}\circ X_{T_n}\circ \theta_{D_B})
           =E^x(\sumnn 1_{E}\circ X_{D_B+{T_n\circ \theta_{D_B}}}),
\end{eqnarray*} 
where, for $P^x$-almost every $\omega\in \Omega$, the last sum
is the number of all $s\ge D_B(\omega)$ such that $X_s(\omega)\in E$, which
in turn is the sum in (\ref{markov}); see (\ref{sge0}).
The equality $v=P^F1_E$ follows taking $B=X$.
\end{proof}

\begin{proof}[Proof of Theorem \ref{main-special}]   
 There  are  Borel  sets   $F_k$ in  $F $     and $a_k\in (0,\infty)$, $k\in\nat$, with
 \begin{equation}\label{choice-Fk}
 1_Uu-1_U\hat u=\sumk a_k 1_{F_k}.
\end{equation}
   (Indeed, choosing step functions    $f_n=\sum_{j=2}^{m_n} a_{nj} 1_{F_{nj}}\in \B(X)$ 
    with $a_{nj}>0$ which are increasing to $f:=1_Uu-1_U\hat u$,  it suffices to observe that
  $f$ is the sum of $f_1$ and the step functions $f_{n+1}-f_n$, $n\in\nat$.)
By Lemma \ref{tech},  $v_k:=P^{F_k}1= P^E1_{F_k}$ for every $k\in\nat$. 
Hence, by Proposition \ref{Rvav}, 
\begin{equation}\label{gg1}
                     g_1:=\sumk a_k v_k \in \R(F) \und g_1=g \on U.
\end{equation} 
So $g\in \N\cap \B(X)$; 
see (\ref{FcapU}). 
 Further,  using  \cite[(2.4) and Proposition~2.3]{HN-mertens}, 
\begin{equation*}
                     g=g_1=\sumk a_k(1_{F_k}+ \hat v_k^f)  =\sumk a_k1_{F_k}+\hat g_1^f=(u-\hat u)+\hat g \on U.    
\end{equation*}

Next let $V\in \U_c$,  $x\in V\cap U$ and $\tau:=D_\vc$.  
By Lemma \ref{tech}, 
\begin{equation*} 
g(x)-P_\tau g(x)=g_1(x)-P_\tau g_1(x)= \sumnn E^x((u-\hat u)\circ X_{S_n}; S_n<\tau)\ge 0
\end{equation*} 
(showing once more that $g$ is nearly hyperharmonic) and, 
  using Lemma~\ref{g-tau},      
\begin{equation*}   
g(x)-P_\tau g(x)\le E^x(u\circ X_{S_0}; S_0<\tau)-\sumnn E^x(u\circ X_\tau; S_n<\tau\le S_{n+1}),
\end{equation*} 
where the sum is equal to $P_\tau u(x)-E^x(u\circ X_\tau;\tau\le S_0)$. Therefore
\begin{equation*} 
 u_0(x)-P_\tau u_0(x)=u(x)-P_\tau u(x) -(g(x)-P_\tau g(x))
 \ge u(x)- E^x(u\circ X_{S_0\wedge \tau}).
 \end{equation*} 
 Since $S_0\wedge \tau=D_{F\cup \vc}$ and $P_{F\cup\vc}u\le u$, 
 by (\ref{nearly-used}), we see  that
 $u_0(x)-P_\tau u_0(x)\ge 0$, that is, $u_0\in \N$.

  Since   $\hat g+\hat u_0=\widehat{g+u_0}=\hat u$,
  by \cite[Proposition 2.3]{HN-mertens}, we finally conclude that  
$\hat u_0=u_0$ on $U$ 
 and  $u_1:=g_1+\hat u_0\in \R(F)$,  $u_1=g+u_0=u$ on~$U$.
\end{proof}

\begin{corollary}\label{corollary-special}
If $w\in\es$ such that $w=\infty$ on $\{u=\infty\}$, then $u+w\in \R(F)$.
\end{corollary} 

\begin{proof} 
By Theorem \ref{main-special}, $u+w=u_1+g_1+w\in \R(F)$.
\end{proof}

\section{The general case}\label{sec-general}

We first reduce the general case of a nearly hyperharmonic function $u\in \B^+(X)$
to the special one considered in the previous section.

\begin{proposition}\label{key-reduction}    
Let $u\in\N\cap \B(X)$ with $\inf u(X) >0$. 
Further,  let $\eta\in (0,1)$,
\begin{equation*}
               F:=\{\hat u< \eta u\} \und v:=1_Fu+1_{X\setminus F} \hat u.
             \end{equation*}
             Then  $\eta u\le v\le u$, $v\in \N\cap \B(X)$, $F=\{\hat v<v\}$    
             and {\bfseries $P^F1< \infty$ on $\{\hat v<\infty\}$. }
\end{proposition} 
 
\begin{proof} Of course, $\hat u\le v\le u$, hence  $\hat v=\hat u$
and  $v\in \N$, by~\cite[Proposition 2.2]{HN-mertens}.  
Clearly, $F=\{\hat v<v\}$ 
and   $a:=\inf \hat v(X)=\inf u(X)>0$.

  The set $ A:=\{\hat u\le \eta u\}$ containing $F$     
is finely closed, since $u$ is finely upper semicontinuous, by Lemma \ref{fusc}.  So, for~every $x\in X$, 
the measure~$P_F(x,\cdot)$ is supported by  $A$,   and hence
\begin{equation*} 
    P_F\hat u(x)  \le \eta P_F u(x)\le \eta u(x).
\end{equation*} 
By regularization, 
$         \hat P_F\hat u\le \eta\hat u$  (see (\ref{PPhat})), that is, $\hat P_F\hat v\le \eta \hat v$.  
  By induction, for every $n\ge 1$, $(\hat P_F)^n\hat v\le \eta^n \hat v$, and therefore
\begin{equation*}
  a P_F(\hat P_F)^n1\le P_F(\hat P_F)^n\hat v\le \eta^n P_F\hat v\le \eta^n\hat v.
  \end{equation*} 
 Thus $P^F1<\infty$ on $\{\hat v<\infty\}$. 
 \end{proof}

Let us say that a   function $u\in \N$ has the \emph{finiteness property} (FP) if, for 
every $x\in X$ with $u(x)<\infty$, there exists a~function $w\in \es$ such that $w=\infty$ on 
$\{u=\infty\}$ and $w(x)<\infty$.  
Trivially, every $u\in\N$ with $u<\infty$ has this property (take $w=0$). 

\begin{theorem}\label{general-1}
Let $u\in \N$ be Borel measurable satisfying {\rm (FP)}. Then $u=R_u$.
\end{theorem} 

\begin{proof}  
Let $x\in X$, $\eta\in (0,1)$ and $\ve>0$.  Of course, $u+\ve\in\N$ and  $u+\ve$ satisfies~(FP). 
By Proposition \ref{key-reduction}, there exists $v\in \N\cap \B(X)$ satisfying Assumption \ref{ass} 
and such that 
$\eta(u+\ve)\le v\le u+\ve$.  
 We  choose $w_1\in \es$ such that $w_1=\infty$ on~$\{u=\infty\}$ and $w_1(x)<\ve$. 
Then $v+w_1\in \R(\{\hat v<v\})$,  by Corollary \ref{corollary-special}. Thus, by (\ref{NR}), 
   $\eta R_u\le  R_{v+w_1}=v+w_1\le u+\ve+w_1$. In particular, $\eta R_u(x)\le u(x)+2\ve$.
\end{proof} 

The following consequence of Theorem \ref{general-1} may be surprising. 
Its combination with Theorem \ref{general-1}
establishes  the implication (1)\,$\Rightarrow$\,(2) in Theorem \ref{main}.

\begin{corollary}\label{u-infty-infty}
 Every Borel measurable $u\in\N$ has the property {\rm (FP)}.
\end{corollary} 

\begin{proof}  
Let $u\in \N$ be Borel measurable, 
 $x\in X$ with $u(x)<\infty$ and  $E:=\{u=\infty\}$. 
Clearly, $1_E=1\wedge \inf_{n\in\nat} (u/n)\in \N$.
By Theorem~\ref{general-1}, there are $w_n\in\es$, $n\in\nat$, such that 
 \begin{equation*} 
 w_n\ge 1_E   \und w_n(x)<2^{-n}. 
 \end{equation*} 
Then $w:=\sumn w_n\in\es$, $w=\infty$ on $E$  and $w(x)<1$. 
\end{proof}

To prove the additional statement in Theorem \ref{main}, let us consider $\vp\in \B^+(X)$   
and recall that $N_\vp:=\inf\{u\in \N\colon u\ge \vp\}$ is the smallest nearly hyperharmonic
majorant of $\vp$, $N_\vp=\vp\vee \hat N_\vp\in \B^+(X)$ (see \cite[Proposition 2.4]{HN-mertens}),
and hence 
\begin{equation*} 
                              N_\vp=\inf\{w\in\es\colon w\ge N_\vp\} \ge \inf \{w\in \es\colon  w\ge \vp\}=R_\vp\ge N_\vp.
\end{equation*}

\section{The remaining part of Theorem \ref{main}}\label{remaining}

For a proof of the implication (1)\,$\Rightarrow$\,(3) in Theorem \ref{main} we note that,
for every $u\in\N\cap \B(X)$, the set $\{\hat u<u\}$ is semipolar   
(which, for example, follows from $u=\inf\{w\in \es\colon w\ge u\}$)
 and that every  semipolar Borel set is the union of   compacts $K_n$, $n\in\nat$,
and a~polar set $E\in \B$; see~\cite[Proposition 5.2]{HN-mertens} in connection with Remark \ref{erratum} below. 
We start with a lemma on~compact sets which will quickly lead to the basic approximation result in Corollary \ref{NRmu}.

\begin{lemma}\label{mu-K-proposition}
Let $K$ be a compact in $X$ and let $w$ be a bounded function in $\es$. 
Then there exists a decreasing sequence $(V_n)$ of finely open Borel sets containing~$K$ such that 
$P_Kw=\inf_{n\in\nat} P_{V_n}w$.
\end{lemma} 

\begin{proof} 
 By \cite[VI.1.2 (and its proof)]{BH}, $P_Kw$ is the infimum of all functions 
$P_Vw$, where~$V$ is a finely open Borel set containing $K$.    By a topological  lemma of Choquet
(see \cite[I.1.8]{BH}), there is
a sequence $(V_n)$ of such sets satisfying
\begin{equation*}
                                         \hat P_Kw= \widehat{\inf_{n\in\nat} P_{V_n}w}.
\end{equation*} 
Fixing a decreasing sequence $(U_n)$  of open sets in $X$ with $\bigcap_{n\in\nat} \ov U_n=K$,
 we may assume without loss of generality that $V_{n+1}\subset V_n\subset U_n$ for every  $n\in\nat$. 
Then, by \cite[VI.2.6]{BH}, every function $P_{V_n}w$, $n\in\nat$, is harmonic on~$X\setminus \ov U_n$. 
Hence  $\inf_{n\in\nat} P_{V_n}w$ is harmonic on $X\setminus K$, by \cite[III.3.1]{BH}, and we obtain that
\begin{equation*}
                             P_Kw=\hat P_Kw = \inf P_{V_n}w \on X\setminus K.
\end{equation*} 
The proof is completed observing that  $P_Kw=w=\inf P_{V_n}w$ on $K$.
\end{proof} 

\begin{proposition}\label{mu-proposition}
Let $A$ be the union of compacts $K_n$ in $X$, $n\in\nat$, and a~polar set $E\in \B$. 
Further, let $w\in\es$ and $\mu\in \M(X)$ be such that  $w $ is bounded and $\mu(X)<\infty$. 
Then 
\begin{equation*} 
\int P_Aw\, d\mu= \inf\{\int P_Vw\,d\mu\colon A\subset V,\  V \mbox{ finely open Borel\,}\} .
\end{equation*} 
\end{proposition} 

\begin{proof} Of course, we may suppose that the sequence $(K_n)$ is increasing. 
Moreover, we may assume that $\mu(A)=0$, since  $P_Aw=w=P_Vw$,  whenever $A\subset V\subset X$. 

Let us fix $\ve>0$. 
Since $T_E=\infty$ a.s., we have $P_E1=0$ on $X\setminus E$. Hence, by \cite[VI.1.9]{BH}, 
there exists an open neighborhood $U$ of $E$ such that $\int P_U1\,d\mu<\ve$. 
Moreover, by~Lemma~\ref{mu-K-proposition}, there exist finely open $V_n\in\B$, $n\in\nat$,
such that 
\begin{equation*}
K_n\subset V_n \und \int P_{V_n}w\,d\mu<\int P_{K_n}w\,d\mu+2^{-n} \ve.
\end{equation*} 
Defining $W_n:=V_1\cup \dots\cup V_n$ and proceeding as in the proof of \cite[VI.1.4]{BH}
 we get
\begin{equation*}
                                 \int P_{W_n}w\,d\mu \le \int P_Kw\,d\mu+(1-2^{-n})\ve \qquad\mbox{ for every }n\in\nat.
\end{equation*} 
Let $W:=\bigcup_{n\ge 1}W_n$ and $V:=W\cup U$. Then $P_Vw\le P_Ww+P_Uw$ and $P_{W_n}w\uparrow P_Ww$,
by \cite[VI.1.7]{BH}. Thus we finally conclude that $\int P_Vw\,d\mu\le \int P_Kw\,d\mu+2\ve$.
\end{proof}

\begin{remark}{\rm  
In \cite[p.\,138]{mertens} such a result is shown for   sets $A$ which are \emph{strictly thin}, that is,
satisfy $\hat P_A1< \eta$ on~$A $ for some $\eta\in (0,1)$, giving first a (delicate)  proof for 
the following stunning approximation of the hitting time $T_A$: For every probability measure $\nu$
on $X$ not charging $A$, there exists  a decreasing sequence $(V_n)$ of finely open   sets containing $A$ 
such that $\limn P^\nu[T_{V_n}<T_A]=0$.
}
\end{remark}

\begin{corollary}\label{NRmu}
Let $u\in \N\cap \B(X)$,   $v\in\R(\{\hat u<u\})$ and let  $\mu\in\M(X)$ be a~finite measure.  Then
 \begin{equation}\label{NRmu-formula}
\int v\,d\mu:=\inf\{\int w\,d\mu\colon w\in\es,\,w\ge v\}.
\end{equation} 
\end{corollary} 

\begin{proof} By (\ref{red-es}), $P_Vw\in \es$ for all  finely open sets $V\in \B$ and   $w\in \es$.        
Thus (\ref{NRmu-formula}) follows immediately from  Proposition \ref{mu-proposition} and Lemma \ref{sums}. 
\end{proof}

\begin{definition}
For every Borel measurable  $u\in\N$,   let $\M_u(X)$ denote the set of all $\mu\in \M(X)$ 
such that $u$ is $\mu$-integrable and  $\mu(A) +\int_{X\setminus A} w\,d\mu<\infty$ 
for some Borel set $A$ in $X$  and majorant $w\in \es$ of $u$. 
\end{definition}

Let  us say that a Borel measurable  $u\in \N$ has the \emph{finiteness property} (FP$'$) if, for
every $\mu\in\M_u$, there exists a~function $w\in \es$ with $w=\infty$ on the
set $\{u=\infty\}$ and~$\int w\,d\mu<\infty$.
Trivially, every $u\in\N$ with $u<\infty$ has this property (take $w=0$).

\begin{theorem}\label{Borel-mu}
Let $u\in \N $ be Borel measurable and  $\mu\in\M_u(X)$.   
If  $u<\infty$ or, more generally, if $u$ has the property~{\rm (FP$'$)},  then 
\begin{equation*}
       \int u\,d\mu=\inf \{\int w\,d\mu\colon w\in\es,\,w\ge u\} . 
\end{equation*} 
\end{theorem}

\begin{proof} Let $\eta\in (0,1)$ and $\ve>0$.
Assuming that $u$ has the property (FP$'$),  we choose $w_1\in \es$ 
such that $w_1=\infty$ on $\{u=\infty\}$ and $\int w_1\,d\mu<\ve$.
By Remark~\ref{KK},  there exists $A\in \B$ and $w_0\in\es$ such that
$w_0\ge u$ and $\mu(A)<\infty$, $\int_{X\setminus A} w_0\,d\mu<\ve$. 
We fix  $\delta>0$ such that $\delta \mu(A) <\ve$, and define  
\begin{equation*}
                 \nu:=1_A \mu  \und u_1:=u+\delta. 
\end{equation*} 

By Proposition \ref{key-reduction}  and Corollary \ref{corollary-special},   
there exists a~Borel measurable   $v\in \N$ such that 
$\eta u_1\le v\le u_1$ and $v+w_1\in\R(\{\hat v<v\})$. Then, by Corollary \ref{NRmu}, 
\begin{eqnarray*} 
& &   \int u\,d\mu +2\ve >\int (u_1+w_1)\,d\nu  \ge        \int (v+w_1)\,d\nu\\
&=&\inf\{\int w\,d\nu\colon  w\in\es,\ w\ge v+w_1\}
\ge \eta \inf\{\int w\,d\nu\colon  w\in\es,\ w\ge u\}. 
\end{eqnarray*} 
So there exists $w\in\es$ such that $w\ge u$ and $\eta \int w\,d\nu<\int u\,d\mu+2\ve$.
We may assume without loss of generality that $w\le w_0$. Then
\begin{equation*}
           \eta   \int w\,d\mu< \eta \int w\,d\nu+\ve<\int u\,d\mu+3\ve.
\end{equation*} 
Letting $\ve$ tend to $0$ and $\eta$ tend to $1$ the proof is completed.
\end{proof}

\begin{corollary}\label{u-infty}
 Every  Borel measurable $u\in\N$ has the property {\rm (FP$'$)}.   
\end{corollary} 

\begin{proof} 
Let $u\in \N$ be Borel measurable, $\mu\in\M_u(X)$ and  $E:=\{u=\infty\}$. 
In particular, $\int u\,d\mu<\infty$,   and hence $\mu(E)=0$. 
Obviously, $1_E=1\wedge \inf_{n\in\nat} (u/n)\in \N$.
Hence, by Theorem~\ref{Borel-mu}, there exist functions 
$w_n\in\es$, $n\in\nat$,  such that
 \begin{equation*} 
w_n\ge 1_E  \und \int w_n\, d\mu<\int 1_E\,d\mu+ 2^{-n}=2^{-n}. 
\end{equation*} 
Then~$w:=\sumn w_n\in\es$, $w=\infty$ on $E$  and $\int w \,d\mu<1$. 
\end{proof} 

Combining Corollary \ref{u-infty} with Theorem \ref{Borel-mu} we obtain
 the implication (1)\,$\Rightarrow$\,(3) in Theorem \ref{main}.

\section{Weakening of the measurability assumption}\label{weaker}

In this section we shall consider nearly hyperharmonic functions which may not be Borel measurable.
Let $\bu$ be the $\sigma$-algebra of all ($\B$-)universally measurable sets and, 
as in \cite{HN-mertens}, let~$\tilde \B,\bup$ respectively denote the $\sigma$-algebra of all
sets $A$ in $X$ for which there exists a~set $B$ in $\B,\bu$ respectively such that the
symmetric difference $A \triangle B$ is polar, that is, $\hat P_{A \triangle B} 1=0$.

Let us  observe that, for functions $u\ge 0$ which are $\bup$-measurable, the upper 
integral in (\ref{def-nearly}) may be replaced by the integral, since the measures 
$P_\vc(x,\cdot)=\vx^\vc$, $x\in V\in \U_c$, do not charge polar sets; see \cite[VI.5.6]{BH}.

Let $f$ be a positive function on $X$ which is $\tilde \B$-measurable. 
Since every polar set is contained in  a~Borel polar set and every countable union of polar sets is polar,
there exist $f_1\in  \B^+(X)$  and $f_2\in \B^+(X)$  such that  the set $\{f_2>0\}$ 
is polar and $f_1\le f\le f_1+f_2$.

The following result extends the implication 
(1)\,$\Rightarrow$\,(3) of  Theorem~\ref{main} to functions which are nearly Borel 
measurable, that is, $\tilde \B\cap \B^\ast$-measurable.

\begin{theorem}\label{nearly-Borel}
Let  $u$ be a  nearly Borel measurable function in $\N$.   Then
\begin{equation*}
     \int u\,d\mu=\inf \{\int w\,d\mu\colon w\in\es,\ w\ge u\}
\end{equation*} 
for every $\mu\in \M(X)$   such that $\mu(A)+ \int_{X\setminus A} w\,d\mu<\infty$ for some $A\in \B$
and some majorant $w\in\es$ of $u$. 
\end{theorem} 

\begin{proof} There exist $u_1,u_2\in \B^+(X)$ such that  $ u_1\le u\le u_1+u_2$ and  the set $\{u_2>0\}$ is polar.
Of course, we may assume that $\hat u\le u_1$. 
Let us fix $\mu\in \M(X)$ 
 such that $\mu(A)+\int_{X\setminus A} w\,d\mu<\infty$ for some $A\in \B$ 
and majorant $w\in\es$ of $u$. 
Choosing $v_1, v_2\in \B^+(X)$ such that
$v_1\le u\le v_1+v_2$ and $v_2=0$ $\mu$-a.e., we may 
assume that $u_1\le v_1$ and $v_2\le u_2$.
 Then $v_1\in \N$, by \cite[Proposition 2.2]{HN-mertens}.
Since $\{v_2>0\}$ is polar, we know that $v_2\in \N$, and hence $v:=v_1+v_2\in \N$. Thus, by Theorem \ref{main}, 
\begin{eqnarray*} 
\int u\,d\mu=\int v \,d\mu&=&\inf\{\int w\,d\mu\colon w\in \es, w\ge v\}\\
&\ge& \inf\{\int w\,d\mu\colon w\in \es, w\ge u\}\ge \int u\,d\mu.
\end{eqnarray*} 
\end{proof}

For the implication (1)\,$\Rightarrow$\,(2) we even have the following.

\begin{theorem}\label{tilde-Borel}
Let  $u$ be a $\tilde \B$-measurable function in $\N$.   Then
\begin{equation*}
     u=\inf \{w\in\es\colon w\ge u\}.
    \end{equation*} 
\end{theorem} 

\begin{proof} 
Let us fix $x\in X$. 
There exist $v_1,v_2\in \B^+(X)$ such that $v_1\le u\le v_1+v_2$ and 
the set $\{v_2>0\}$ is polar whence $v_2\in\N$.  Of course, we may assume that $\hat u\le v_1$,
$v_1(x)=u(x)$ and $v_2(x)=0$. By \cite[Proposition 2.2]{HN-mertens}, $v_1\in \N$.
So  $v:=v_1+v_2\in \N$ and  $R_v=v$, by Theorem~\ref{main}. 
Thus $R_u(x)\le R_v(x)=v(x)=  u(x)\le R_u(x)$.
\end{proof}

Using results of \cite[Section 4]{HN-mertens} this 
leads to a characterization of the equality $R_u=u$ for nearly hyperharmonic
$\bup$-measurable functions. 
To this end we  recall   that the
$\sigma$-algebra of all finely Borel subsets of $X$ (that is, the smallest $\sigma$-algebra
 on~$X$ containing all finely open sets) is the smallest $\sigma$-algebra containing~$\B$ and all
semipolar sets; see \cite[Section 5]{HN-mertens}. In particular, $\tilde \B\subset \B^f$.

\begin{theorem}\label{general-2}Let  $u\in\N$  and suppose that $u$ is $\bup$-measurable. 
Then the following statements are equivalent:
\begin{itemize} 
\item[\rm (i)] 
$u=\inf \{w\in \es\colon w\ge u\}$.
\item[\rm (ii)] 
$u$ is finely upper semicontinuous. 
\item[\rm (iii)] 
$u$ is finely Borel measurable.
\item[\rm (iv)] 
$u$ is $\tilde \B$-measurable.
\item[\rm (v)] 
The set $\{\hat u<u\}$ is semipolar. 
\end{itemize} 
\end{theorem}

\begin{proof} Trivially,  (i)\,$\Rightarrow$\,(ii)\,$\Rightarrow$\,(iii).
Moreover,  (iii)\,$\Leftrightarrow$\,(v) and (iii)\,$\Rightarrow$\,(iv), 
by \cite[Proposition 5.3 and Corollary 5.4]{HN-mertens}. By Theorem \ref{tilde-Borel},
(iv)\,$\Rightarrow$\,(i).
\end{proof} 

Clearly, previous statements on reduced functions $R_\vp$ can now  
be  extended to functions $\vp\ge 0$ on $X$ which are only supposed to be  $\B^f\cap \bup$-measurable.

\begin{remark}\label{erratum}{\rm
The result \cite[Corollary 5.4]{HN-mertens}    
relies on \cite[Proposition 5.2]{HN-mertens} 
the proof of which uses \cite[Theorem 1.5]{S-negligible} stating that, given a semipolar set $S$,
there exists a~measure $\mu$ on $X$ such that $\mu^\ast(B)>0$ for every non-polar subset $B$
of $S$. 
This is correct; its proof, however, is not, since \cite[Lemma 1.3]{S-negligible} is wrong.

Assuming without loss of generality that $S$ is the union of totally thin Borel sets~$F_n$, $n\in\nat$,
we many obtain a valid proof  exhausting each $F_n$ by sets  
 \begin{equation*} 
F_{n,m}:=F_n\cap K_m\cap \{ \hat P_{F_n}q<\eta_m q\} , \qquad  m\in\nat,
\end{equation*}  where $q$ is a continuous strict potential on $X$,
 $(K_m)$ is a sequence of compacts  in $X$ and $\eta_m\in (0,1)$ such that $K_m\uparrow X$ and $\eta_m\uparrow 1$
as $m\to\infty$. 

Indeed, let us fix $x\in X$ and $m,n\in \nat$. Let $F:=F_{n,m}$, $\eta:=\eta_{n,m}$. We recursively define
measures $\mu_k$  on $F$  taking $\mu_0:=\vx$ and 
\begin{equation*} 
\mu_k:=\hat\mu_{k-1} ^F:=\int \hat\ve_y^F\,d\mu_{k-1}(y), \qquad k\in\nat.
\end{equation*} 
Then $\int q\,d\mu_k=\int \hat P_Fq\,d\mu_{k-1}\le \eta \int q \, d\mu_{k-1}$, and hence 
$ \inf q(F)\cdot \mu_k(F) \le \eta^k q(x)$  for every $k\in\nat$.
So $\lim_{k\to\infty} \mu_k=0$ which leads to a proof of \cite[Theorem 1.5]{S-negligible} not using 
 \cite[Lemma 1.3]{S-negligible}.
}
\end{remark}

\section{Application of our method to right processes}\label{standard}

 It should be clear to the experts that our approach works as well  
for a general   right process $\mathfrak X$  on a Radon space $X$
provided we assume that the associated potential kernel   is proper.
It would then be convenient to define~$\N$ to be the set of all functions
 $u\colon X\to [0,\infty]$ such that $\int^\ast u\circ X_{D_K}\,dP^x\le u(x)$
 for all $x\in X$ and compacts $K$ not containing $x$. By a result of Dynkin
 (see \cite[Theoem II.5.1]{BG}), nearly Borel measurable
functions in $\N$ are supermedian with respect to the resolvent $(V_\lambda)$ associated
with $\mathfrak X$  and, essentially, we only have to replace the lower semicontinuous 
regularization $\hat u$ of $u$ (which for a nearly hyperharmonic function  in our setting of a~nice 
Hunt process is the greatest excessive minorant of~$u$)
by the excessive regularization  $\tilde u:=\lim_{\lambda\to \infty } \lambda V_\lambda  u$.

{\small \noindent 
Wolfhard Hansen,
Fakult\"at f\"ur Mathematik,
Universit\"at Bielefeld,
33501 Bielefeld, Germany, e-mail:
 hansen$@$math.uni-bielefeld.de}\\
{\small \noindent Ivan Netuka,
Charles University,
Faculty of Mathematics and Physics,
Mathematical Institute,
 Sokolovsk\'a 83,
 186 75 Praha 8, Czech Republic, email:
netuka@karlin.mff.cuni.cz}

\end{document}